\documentclass[reqno]{article}		
\usepackage{amssymb,amsmath,amsthm,enumitem}
\usepackage{pifont}

\setlist{nolistsep}        

\newtheorem{theorem}{Theorem}[section]
\newtheorem{lemma}[theorem]{Lemma}
\newtheorem{corollary}[theorem]{Corollary}
\newtheorem{proposition}[theorem]{Proposition}

\theoremstyle{definition}
\newtheorem{question}{Question}
\newtheorem{example}[theorem]{Example}

\newcommand{\iin}{\mathord\in}
\newcommand{\aproda}{\mathbin\text{\Pisymbol{pxsya}{"03}}}
\newcommand{\aprodb}{\mathbin\text{\Pisymbol{pxsyc}{"5E}}}

\newcommand{\card}[1]{\lvert#1\rvert}
\newcommand{\fset}{\mathcal{F}}
\newcommand{\Ftwo}{\mathbb{F}_2}
\newcommand{\Zint}{\mathbb{Z}}
\newcommand{\SYM}{\mathsf{Sym}}
\newcommand{\FINSYM}{\SYM_{<\aleph_0}(\omega)}
\newcommand{\FINC}{\mathsf{FC}}
\newcommand{\SUPP}{\mathsf{supp}}

\title{DTC ultrafilters on groups}

\author{Jan Pachl \and
Juris Stepr\={a}ns\thanks{Research supported by NSERC.}}

\date{}

\begin{document}

\maketitle

\begin{abstract}
We say that an ultrafilter on an infinite group $G$ is DTC
if it determines the topological centre of the semigroup $\beta G$.
We prove that DTC ultrafilters do not exist for virtually BFC groups,
and do exist for the countable groups that are not virtually FC.
In particular, an infinite finitely generated group is virtually abelian
if and only if it does not admit a DTC ultrafilter.
\end{abstract}


\section{Introduction}

When $G$ is an infinite group,
the binary group operation on $G$ extends to the \v{C}ech--Stone compactification $\beta G$
in two natural ways.
They are defined in section~\ref{sec:prelim} and,
as in~\cite[Ch.6]{Dales2010bas}, denoted by $\aproda$ and $\aprodb$.
Say that $v\iin\beta G$ is a \emph{DTC ultrafilter for $\beta G$} if
$u\aproda v \neq u\aprodb v$ for every $u\iin\beta G \setminus G$.
DTC stands for \emph{determining the (left) topological centre}.

Dales et al~\cite{Dales2010bas} investigate the DTC notion in the context of their study
of Banach algebras on semigroups and their second duals.
They prove that the free group $\Ftwo$ admits
a DTC ultrafilter~\cite[12.22]{Dales2010bas},
and that no abelian group does.
Here we address the problem of characterizing those groups that admit
DTC ultrafilters, the class of groups we call DTC(1).
It is not obvious how good such a characterization can be,
even for countable groups,
as it is not clear from the definition whether
the class of countable DTC(1) groups, suitably encoded,
is even within the projective hierarchy.
However, it follows from our results that the class of finitely generated DTC(1) groups
is a Borel set.
In fact, we prove that an infinite finitely generated group is virtually abelian
if and only if it does not belong to DTC(1).
This provides a partial answer to question~(13)
in~\cite[Ch.13]{Dales2010bas}.

The algebraic property used above to define a DTC ultrafilter
is equivalent to a topological one:
$v\iin\beta G$ is a DTC ultrafilter for $\beta G$ if and only if
for every $u\iin\beta G \setminus G$
the mapping $w\mapsto u\aproda w$ from $G\cup\{v\}$ to $\beta G$
is discontinuous at $v$.
In this paper we do not deal with
another notion of sets determining the topological centre,
in which the mapping $w\mapsto u\aproda w$ from the whole $\beta G$ to $\beta G$
is required to be discontinuous at $v$.
Every infinite group admits a DTC ultrafilter in that sense.
Budak et al~\cite{Budak2011mdt} discuss and compare the two notions.


\section{Preliminaries}
    \label{sec:prelim}

When $X$ is a set, $\beta X$ is the \v{C}ech--Stone compactification of $X$,
the set of ultrafilters on $X$ with the usual compact topology~\cite[\S3.2]{Hindman2012asc}.
We identify each element of $X$ with the corresponding principal ultrafilter,
so that $X\subseteq\beta X$.
When $Y\subseteq X$, identify each ultrafilter on $Y$ with its image on $X$,
so that $\beta Y \subseteq \beta X$.

If $\{x_\xi\}_{\xi\in I}$ is a net of elements of $X$
indexed by a directed partially ordered set $I$,
then $\fset := \{ \{ x_\xi \mid \xi \geq \eta \} \mid \eta \iin I \}$
is a filter of subsets of $X$.
An ultrafilter $u\iin\beta X$ is a cluster point of the net $\{x_\xi\}_{\xi}$
in $\beta X$ if and only if $\fset \subseteq u$.

In accordance with the standard set theory notation,
each ordinal is the set of all smaller ordinals,
and the least infinite ordinal is $\omega=\{0,1,2,\dots\}$.
The cardinality of a set $X$ is $\card{X}$.
The set of all subsets of $X$ of cardinality $\kappa$ is $[X]^\kappa$,
and the set of all subsets of cardinality less than $\kappa$ is
$[X]^{<\kappa}$.

When $G$ is a group, denote by $e_G$ its identity element.
When $G$ is an infinite group and $u,v\iin\beta G$,
define~\cite[Ch.6]{Dales2010bas}
\begin{align*}
u \aproda v & := \{ A \subseteq G \mid \{ x\iin G \mid x^{-1} A \iin v \} \iin u \} \\
u \aprodb v & := \{ A \subseteq G \mid \{ x\iin G \mid A x^{-1} \iin u \} \iin v \}
\end{align*}
The operations $\aproda$ and $\aprodb$ are associative, and
$u \aproda v, u \aprodb v \iin \beta G$ for $u,v\iin\beta G$.
Thus $(\beta G,\aproda)$ and $(\beta G,\aprodb)$ are semigroups.
When $u\iin\beta G$, define $u^{\aproda 1} := u$ and
$u^{\aproda (n+1)} := u \aproda u^{\aproda n}$
for $n\iin\omega$, $n>0$.

Say that $D\subseteq\beta G$ is a \emph{(left) DTC set
for $\beta G$ (in the sense of Dales--Lau--Strauss~\cite{Dales2010bas})} if
\[
\forall u\iin\beta G \setminus G \quad
\exists v\iin D \quad u\aproda v \neq u\aprodb v .
\]
Thus $v\iin\beta G$ is a DTC ultrafilter for $\beta G$
(defined in the introduction) if and only if
the singleton $\{v\}$ is a DTC set.

The next lemma gathers the elementary properties of $\aproda$ and $\aprodb$ needed
in the sequel.

\begin{lemma}
    \label{lem:prelim}
The following hold for any infinite group $G$ and $x,y \iin G$, $u,v\iin \beta G$:
\begin{enumerate}[label=({\roman*})]
\item\label{lem:prelim:i}
$x\aproda y = x \aprodb y = xy$.
\item
$x\aproda u = x\aprodb u$.
\item\label{lem:prelim:iii}
$u\aproda x = u\aprodb x$.
\item\label{lem:prelim:iv}
If $U\iin u$ and $V\iin v$ then $UV\iin u\aproda v$ and $UV\iin u\aprodb v$.
\item
If $H$ is a subgroup of $G$ then
$(\beta H,\aproda)$ is a subsemigroup of $(\beta G,\aproda)$ and
$(\beta H,\aprodb)$ is a subsemigroup of $(\beta G,\aprodb)$.
\item\label{lem:prelim:viii}
If $G$ is abelian then $u\aproda v = v \aprodb u$,
and in particular $u\aproda u = u \aprodb u$.
\item\label{lem:prelim:ix}
If $v$ is a DTC ultrafilter then $v\iin \beta G \setminus G$.
\item\label{lem:prelim:x}
If $u,v\iin\beta G \setminus G$ then $u \aproda v,u \aprodb v\iin\beta G \setminus G$.
\end{enumerate}
\end{lemma}

\begin{proof}
Parts~\ref{lem:prelim:i} to \ref{lem:prelim:viii} follow
directly from the definition of $\aproda$ and $\aprodb$,
and \ref{lem:prelim:ix} follows from \ref{lem:prelim:iii}.
Part~\ref{lem:prelim:x} is a special case of Corollary~4.29 in~\cite{Hindman2012asc}.
\end{proof}

Say that a group is DTC(0) if it is finite.
When $G$ is an infinite group, say $G$ is DTC($\kappa$) if
$\kappa$ is the least cardinality of a DTC set for $\beta G$.
By the next theorem every infinite group is either DTC(1) or DTC(2).

\begin{theorem}
    \label{th:twopt}
Let $G$ be an infinite group.
Then there is a two-point DTC subset of $\beta G$.
\end{theorem}

This is a special case of Theorem~12.15 in~\cite{Dales2010bas}
and of Theorem~1.2 in~\cite{Budak2011mdt}.
Here we give a direct proof using the following lemma.

\begin{lemma}
    \label{lem:separation}
Let $G$ be an infinite group, and let $A$ be an index set with $\card{A}=\card{G}$.
For each $\alpha\iin A$ let $F_\alpha$ be a finite subset of $G$.
Then there are $z_\alpha \iin G$ for $\alpha\iin A$ such that
\begin{equation}
\label{eq:sep}
F_\alpha z_\alpha z_\gamma^{-1} \cap F_\beta z_\beta z_\delta^{-1} = \emptyset
\quad\text{for}\quad \{\alpha,\beta,\gamma,\delta\}\iin [A]^4.
\end{equation}
\end{lemma}

\begin{proof}
Without loss of generality, assume $A$ is the cardinal $\kappa=\card{G}$.
Define $z_0 = z_1 = z_2 = e_G$.
Then proceed by transfinite recursion:
For $\beta\iin\kappa\setminus \{0,1,2\}$,
when $z_\alpha$ have been defined for all $\alpha\iin\beta$,
take any
\[
z_\beta \iin G \setminus
         \bigcup_{\{\alpha,\gamma,\delta\}\in[\beta]^3}
         \left(
             F_\beta^{-1} F_\alpha z_\alpha z_\gamma^{-1} z_\delta
                \cup z_\gamma z_\alpha^{-1} F_\alpha^{-1} F_\delta z_\delta
        \right)\;.
\qedhere
\]
\end{proof}

\begin{proof}[Proof of Theorem~\ref{th:twopt}]
Put $A:=\{0,1\}\times[G]^{<\aleph_0}$ and $F_{iK}:=K$
for every $(i,K)\iin A$.
By Lemma~\ref{lem:separation} there are $z_{iK}$ for $(i,K)\iin A$
such that (\ref{eq:sep}).
For $i=0,1$ the elements $z_{iK}$ form a net indexed by the directed poset
$([G]^{<\aleph_0},\subseteq)$;
let $v_i \iin \beta G$ be a cluster point of the net $\{ z_{iK} \}_K$.
We shall prove that $\{v_0 , v_1 \}$ is a DTC set.

For $i=0,1$ define
\begin{align*}
W_i := & \bigcup \{K z_{iK} \mid K \iin [G]^{<\aleph_0} \} \\
S_i := & \bigcup \{K z_{iK} z_{iL}^{-1} \mid K,L \iin [G]^{<\aleph_0}, K \neq L \}
\end{align*}

Since $x^{-1} W_i \iin v_i$ for every $x\iin G$,
it follows that $W_i\iin u \aproda v_i$ for every $u\iin\beta G$.

Now take any $u\iin \beta G \setminus G$.
From (\ref{eq:sep}) we have $S_0 \cap S_1 = \emptyset$,
so there is $i\iin \{0,1\}$ such that $S_i \not\in u$.
For every $K \iin [G]^{<\aleph_0}$ we have
$W_i z_{iK}^{-1} \subseteq K \cup S_i $, hence $W_i z_{iK}^{-1} \not\in u$.
But $\{ z_{iK} \mid K \iin [G]^{<\aleph_0} \} \iin v_i$,
and from the definition of $\aprodb$ we get $W_i \not\in u\aprodb v_i$.
We have proved that $u \aproda v_i \neq u\aprodb v_i$.
\end{proof}

As will be seen in sections~\ref{sec:suff} and~\ref{sec:countable},
properties DTC($\kappa$), $\kappa=1,2$, are connected to the structure of conjugacy classes.
For any group $G$  and $y,z\iin G$ define
\begin{align*}
G_{yz} &:= \{ x \iin G \mid x^{-1}yx=z \} \\
[y]_G  &:= \{ x^{-1}yx \mid x\iin G\}       \\
\FINC(G) &:= \{ y\iin G \mid [y]_G \text{ is finite} \}
\end{align*}
Each $G_{yy}$ is a subgroup of $G$.
Clearly $G_{yz}\neq\emptyset$ if and only if $[y]_G=[z]_G$,
and in that case $G_{yz}$ is a right coset of $G_{yy}$:
If $x^{-1}yx=z$ then $G_{yz}=G_{yy} x$.
For a fixed $y\iin G$,
$\varphi(x):= x^{-1}yx$ defines a mapping $\varphi$ from $G$ onto $[y]_G$
such that $\varphi^{-1}(z)=G_{yz}$ for each $z\iin[y]_G$.
Hence the cardinality of $[y]_G$ equals the index of $G_{yy}$ in $G$.

$\FINC(G)$ is a normal subgroup of $G$.
Say that $G$ is an \emph{ICC group} if $G$ is infinite and $\FINC(G)=\{e_G \}$.
Say that $G$ is an \emph{FC group} if $\FINC(G)=G$.
Say that $G$ is a \emph{BFC group} if $\sup_{y\in G} \card{[y]_G} < \infty$.

When $P$ is a property of groups, a group is said to be \emph{virtually $P$} if
it has a subgroup of finite index that has property $P$.

In section~\ref{sec:countable} we need the following known results,
which are respectively Lemma~4.17, Theorem~1.41 and a corollary of Theorem~4.32
in~\cite{Robinson1972fcg}.

\begin{samepage}
\begin{lemma}
    \label{lem:threelem}
\begin{enumerate}[label=({\roman*})]
\item\label{lem:neumann}
(B.H.~Neumann's theorem)
Let $G$ be a group such that $G=\bigcup_{i=0}^n x_i G_i$ where $x_i \iin G$
and $G_i$ is a subgroup of $G$ for $i=0,1,\dots,n$.
Then at least one of the groups $G_i$ has finite index in~$G$.
\item\label{lem:schreier}
(Schreier's subgroup lemma)
Every subgroup of finite index in a finitely generated group is finitely generated.
\item\label{lem:fingenvab}
Every finitely generated FC group is virtually abelian.
\end{enumerate}
\end{lemma}
\end{samepage}


\section{Sufficient condition}
    \label{sec:suff}

In this section we prove a sufficient condition for a group to be DTC(2).
We start by showing that the DTC(1) and DTC(2) properties are inherited by subgroups
of finite index.

\begin{theorem}
    \label{th:finindex}
Let $G$ be a group and $H$ its subgroup of finite index.
Then $G$ is DTC(1) if and only if $H$ is.
\end{theorem}

\begin{proof}
As $G$ is the finite union of the left cosets of $H$,
for every ultrafilter $u\iin\beta G$ there is $y\iin G$ such that $yH\iin u$,
and then $y^{-1}\aproda u = y^{-1}\aprodb u \iin \beta H$.

When $G$ is DTC(1), let $v\iin \beta G$ be a DTC ultrafilter for $\beta G$,
and let $y\iin G$ be such that $y^{-1}\aproda v = y^{-1}\aprodb v \iin \beta H$.
Take any $u\iin\beta H \setminus H$.
Then $u \aproda y^{-1} = u \aprodb y^{-1} \iin \beta G \setminus G$ and
\[
u \aproda (y^{-1}\aproda v)
= (u \aproda y^{-1}) \aproda v
\neq (u \aprodb y^{-1}) \aprodb v
= u \aprodb (y^{-1}\aprodb v) .
\]
Thus $y^{-1}\aproda v$ is a DTC ultrafilter for $\beta H$, and $H$ is DTC(1).

When $H$ is DTC(1),
let $v\iin\beta H \subseteq \beta G$ be a DTC ultrafilter for $\beta H$.
Take any $u\iin\beta G \setminus G$,
and let $y\iin G$ be such that
$y^{-1}\aproda u = y^{-1}\aprodb u \iin \beta H \setminus H$.
Then
\[
y^{-1}\aproda (u\aproda v)
= (y^{-1}\aproda u)\aproda v
\neq (y^{-1}\aprodb u)\aprodb v
= y^{-1}\aprodb (u\aprodb v)
= y^{-1}\aproda (u\aprodb v),
\]
hence $u\aproda v \neq u\aprodb v$.
Thus $v$ is a DTC ultrafilter for $\beta G$, and $G$ is DTC(1).
\end{proof}

By~\ref{lem:prelim}\ref{lem:prelim:viii} and~\ref{th:finindex},
every infinite virtually abelian group is DTC(2).
Next we shall prove that even every infinite virtually BFC group is DTC(2).

\begin{lemma}
    \label{lem:finhom}
Let $G$ be an infinite group, $H$ a finite group,
and $\alpha\colon G \to H$ a surjective homomorphism.
Then $\alpha$ extends to a homomorphism
$\overline{\alpha}\colon (\beta G,\aproda)\to H$ such that
for every $u\iin\beta G$ and every $h\iin H$ we have
$\alpha^{-1}(h)\iin u$ if and only if $\overline{\alpha}(u)=h$.
\end{lemma}

In fact $\overline{\alpha}$ is the unique continuous extension of $\alpha$ to $\beta G$.
This observation is not needed in the sequel.

\begin{proof}
Write $\overline{\alpha}(u):=h$ when
$u\iin\beta G$, $h\iin H$ and $\alpha^{-1}(h)\iin u$.
That defines a mapping from $\beta G$ onto $H$,
because the sets $\alpha^{-1}(h)$, $h\iin H$, form a finite partition of $G$,
and therefore for every $u\iin\beta G$ there is a unique $h\iin H$ such that
$\alpha^{-1}(h)\iin u$.

Obviously $\overline{\alpha}(x)=\alpha(x)$ for $x\iin G$.
To prove $\overline{\alpha}$ is a homomorphism, take any $u,v\iin\beta G$
and let $f:=\overline{\alpha}(u)$, $h:=\overline{\alpha}(v)$.
Then $\alpha^{-1}(fh) = \alpha^{-1}(f) \alpha^{-1}(h) \iin u \aproda v$
by \ref{lem:prelim}\ref{lem:prelim:iv},
hence $\overline{\alpha}(u \aproda v)=fh$.
\end{proof}

\begin{theorem}
    \label{th:suff}
Let $G$ be an infinite FC group for which there exists $n\iin\omega$, $n\geq 1$,
such that $x^n y = y x^n$ for all $x,y\iin G$.
Then $G$ is DTC(2).
\end{theorem}

\begin{proof}
First we shall prove that
\begin{equation}
\label{eq:xny}
v^{\aproda n}\aprodb u = u \aproda v^{\aproda n} \quad\text{for all}\quad
u,v\iin\beta G.
\end{equation}

Take any $y\iin G$.
Denote by $\SYM([y]_G)$ the group of all permutations of the finite set $[y]_G$.
Define the homomorphism $\alpha\colon G \to \SYM([y]_G)$ by
\[
\alpha(x)(z):=x^{-1}zx, \quad x\iin G, z\iin [y]_G,
\]
and let $H:=\alpha(G)\subseteq \SYM([y]_G)$.
Write $E:=\alpha^{-1}(e_H)$.
Then $xy=yx$ for $x\iin E$, hence
\[
E\cap y^{-1}A = E\cap Ay^{-1}
\]
for every $A\subseteq G$.

Take any $v\iin\beta G$ and $A\subseteq G$.
Let $\overline{\alpha}\colon (\beta G,\aproda)\to H$
be the extension of $\alpha$ as in Lemma~\ref{lem:finhom}.
As $x^{-n} z  x^n = z$ for all $x\iin G$ and $z\iin [y]_G$, we get
$h^n=e_H$ for every $h\iin H$.
Hence $\overline{\alpha}(v^{\aproda n}) = \overline{\alpha}(v)^n = e_H$,
and thus $E \iin v^{\aproda n}$.
It follows that $y^{-1}A \iin v^{\aproda n}$ if and only if
$E\cap Ay^{-1} = E \cap y^{-1}A \iin v^{\aproda n}$
if and only if $Ay^{-1} \iin v^{\aproda n}$.

We have proved
\[
y^{-1}A \iin v^{\aproda n} \Leftrightarrow Ay^{-1} \iin v^{\aproda n}
\]
for all $y\iin G$, $v\iin\beta G$ and $A\subseteq G$.
Now~(\ref{eq:xny}) follows from the definition of $\aproda$ and $\aprodb$.

Next take any $v\iin\beta G \setminus G$.
Then $v^{\aproda n}\iin\beta G \setminus G$ by~\ref{lem:prelim}\ref{lem:prelim:x},
and from~(\ref{eq:xny}) we get
\[
v^{\aproda n}\aproda v = v^{\aproda (n+1)}
= v \aproda v^{\aproda n} = v^{\aproda n} \aprodb v .
\qedhere
\]
\end{proof}

\begin{corollary}
    \label{cor:virtBFC}
Every infinite virtually BFC group is DTC(2).
\end{corollary}

\begin{proof}
In view of Theorem~\ref{th:finindex} it is enough to prove that every
infinite BFC group is DTC(2).
For any such $G$,
apply Theorem~\ref{th:suff}
with $n$ equal to the factorial of $\max_{y\in G} \card{[y]_G}$.
\end{proof}

In Example~\ref{ex:redprod} we produce a countable group that is DTC(2) but not
virtually BFC.


\section{Countable groups}
    \label{sec:countable}

Let $G$ be a countable infinite group,
$G= \bigcup_{n\in\omega} F_n$ where $F_0\subseteq F_1\subseteq F_2 \subseteq \dots$
are finite sets.
Let $v\iin\beta G$ be a cluster point of a sequence $\{x_n\}_{n\in\omega}$ in $G$,
and set $W:= \bigcup_n F_n x_n$.
From the definition of $\aproda$ we get
$W \iin u \aproda v$ for every $u\iin\beta G$.
The next theorem describes a condition that allows a choice of $x_n$ for which
$W \not\in u \aprodb v$ for every $u\iin\beta G \setminus G$,
so that $v$ is a DTC ultrafilter.

\pagebreak
\begin{theorem}
    \label{th:countable}
Consider four conditions for a countable infinite group $G$:
\begin{enumerate}[label=({\roman*})]
\item\label{th:countable:conda}
\nopagebreak
There is $V\subseteq G$ such that for every finite $F\subseteq G$
there is $x\iin G$ for which
\[
x\not\in F V \cup F x (F \setminus V) .
\]
\item\label{th:countable:condb}
There are finite $F_n\subseteq G$, $n\iin\omega$,
such that $e_G \iin F_n = F^{-1}_n$
and $F_n F_n \subseteq F_{n+1}$ for all $n$, and
$G= \bigcup_{n\in\omega} F_n$.
There is a sequence $\{x_n\}_{n\in\omega}$ in $G$ such that
\begin{align}
\label{eq:a1}
F_n x_n x^{-1}_i \cap F_k x_k x^{-1}_j & = \emptyset \quad\text{for}\;\; i,j< k < n \\
\label{eq:a2}
F_n x_n x^{-1}_i \cap F_n x_n x^{-1}_j & = \emptyset \quad\text{for}\;\; i < j < n
\end{align}
\item\label{th:countable:condc}
There are $v\iin\beta G \setminus G$ and $W\subseteq G$ such that
$W\iin u\aproda v$ for all $u\iin\beta G$ and
$W\not\in u\aprodb v$ for all $u\iin\beta G \setminus G$.
\item\label{th:countable:condd}
$G$ is DTC(1).
\end{enumerate}
Then \ref{th:countable:conda}$\Rightarrow$\ref{th:countable:condb}%
$\Rightarrow$\ref{th:countable:condc}$\Rightarrow$\ref{th:countable:condd}.
\end{theorem}

\begin{proof}
Write $G= \bigcup_{n\in\omega} F_n$ with finite sets $F_n\subseteq G$
such that $e_G \iin F_n = F^{-1}_n$ and $F_n F_n \subseteq F_{n+1}$ for all $n$.
Assuming \ref{th:countable:conda},
$e_G \iin V$ because otherwise we would have
$x\iin \{e_G\} x (\{e_G\} \setminus V)$ for all $x\iin G$.
A recursive construction yields a sequence of $x_n$ such that
\begin{align}
\label{eq:b1}
x_n & \not\in x_i V \cup F_{n+1} x_k x^{-1}_i x_j \quad\text{for}\;\; i,j < k < n \\
\label{eq:b2}
x_n & \not\in F_{n+1} x_n x^{-1}_i x_j \quad\quad\quad\;\;\text{for}\;\; i < j < n
\end{align}
Since (\ref{eq:a1}) follows from (\ref{eq:b1}) and (\ref{eq:a2}) follows from (\ref{eq:b2}),
this proves \ref{th:countable:conda}$\Rightarrow$\ref{th:countable:condb}.

Now assume \ref{th:countable:condb}.
From (\ref{eq:a2}) we get $x_i \neq x_j$ for $i\neq j$.
Let $v\iin\beta G \setminus G$ be a cluster point of the sequence $\{x_n\}_n$,
and put $W:= \bigcup_n F_n x_n$.
Then $W \iin u \aproda v$ for every $u\iin\beta G$.

Take any $u\iin\beta G \setminus G$.
From (\ref{eq:a1}) and (\ref{eq:a2}), for $i<j$ we get
\[
W x^{-1}_i \cap W x^{-1}_j
= \left( \bigcup_{n\in\omega} F_n x_n x^{-1}_i \right)
\cap \left( \bigcup_{k\in\omega} F_k x_k x^{-1}_j \right)
\subseteq \bigcup_{k=0}^j F_k x_k x^{-1}_j
\]
Thus the intersection $W x^{-1}_i \cap W x^{-1}_j$ is finite for $i\neq j$.
Hence there is at most one $i\iin\omega$ for which $W x^{-1}_i \iin u$.
Since $v\iin\beta G \setminus G$ and $\{x_i \mid i\iin\omega\} \iin v$,
it follows that $\{ x \mid W x^{-1} \iin u \} \not\in v$,
and $W\not\in u \aprodb v$ from the definition of $\aprodb$.
That proves \ref{th:countable:condb}$\Rightarrow$\ref{th:countable:condc}.

Obviously \ref{th:countable:condc}$\Rightarrow$\ref{th:countable:condd}.
\end{proof}

We do not know if condition~\ref{th:countable:condd} in Theorem~\ref{th:countable}
is inherited from quotients, but condition~\ref{th:countable:conda} is:

\begin{proposition}
    \label{prop:quotient}
Let condition~\ref{th:countable}\ref{th:countable:conda} hold for a group $G$
Let $H$ be a group with a surjective homomorphism $\pi\colon H \to G$.
Then condition~\ref{th:countable}\ref{th:countable:conda} holds also for $H$ in place of $G$
and $\pi^{-1}(V)$ in place of $V$.
\end{proposition}

\begin{proof}
Write $U:=\pi^{-1}(V)$.
Then $\pi(F\setminus U) = \pi(F) \setminus V$ for every $F\subseteq H$.

Take any finite $F\subseteq H$.
By the assumption there is $x\iin G$ such that
$x\not\in \pi(F) V \cup \pi(F) x (\pi(F) \setminus V)$.
Let $y\iin H$ be such that $\pi(y)=x$.
Then
\[
\pi(y) \not\in \pi(F) \pi(U) \cup \pi(F) \pi(y) \pi(F \setminus U),
\]
hence $y\not\in F U \cup F y (F \setminus U)$.
\end{proof}

\begin{theorem}
    \label{th:FCinf}
Let $G$ be a countable group such that $\card{G/\FINC(G)}=\aleph_0$.
Then $G$ satisfies condition~\ref{th:countable}\ref{th:countable:conda},
and hence is DTC(1).
\end{theorem}

\begin{proof}
We shall prove that $G$ satisfies condition~\ref{th:countable}\ref{th:countable:conda}
with $V=\FINC(G)$.
Take any finite set $F\subseteq G$.
If $y\iin V$ and $z\iin F\setminus V$ then $G_{yz}=\emptyset$,
and for $y\iin F^{-1}\setminus V$ the index of $G_{yy}$ is infinite.
Therefore by~\ref{lem:threelem}\ref{lem:neumann} there is
\[
x \not\in FV \cup \bigcup_{y\in F^{-1}}\;\;
\bigcup_{z\in F\setminus V} G_{yz}
\]
which means that $x\not\in FV \cup F x (F\setminus V)$.
\end{proof}

\begin{corollary}
    \label{cor:icc}
Every countable ICC group satisfies
condition~\ref{th:countable}\ref{th:countable:conda}, and hence is DTC(1).
\end{corollary}

\begin{corollary}
    \label{cor:iccquot}
Every countable group that has an ICC quotient is DTC(1).
\end{corollary}

\begin{proof}
Apply Proposition~\ref{prop:quotient} and Corollary~\ref{cor:icc}.
\end{proof}

\begin{corollary}
    \label{cor:vnilp}
An infinite finitely generated group is DTC(2) if and only if it is virtually abelian.
\end{corollary}

\begin{proof}
Let $G$ be an infinite finitely generated group.
If $G$ is virtually abelian then it is DTC(2) by Theorem~\ref{th:finindex}.
If $G$ is DTC(2) then $\card{G/\FINC(G)}<\aleph_0$ by Theorem~\ref{th:FCinf}.
In that case $\FINC(G)$ is finitely generated by~\ref{lem:threelem}\ref{lem:schreier},
hence $G$ is virtually abelian by~\ref{lem:threelem}\ref{lem:fingenvab}.
\end{proof}

Example~\ref{ex:redprod} shows that the assumption that the group is finitely generated
cannot be omitted in Corollary~\ref{cor:vnilp}.


\section{Examples}

Dales et al~\cite[12.22]{Dales2010bas} prove that the free group $\Ftwo$ is DTC(1).
This follows from Corollary~\ref{cor:icc}, since non-commutative free groups are
ICC~\cite[Ex.8.3]{Ceccherini2010cag}.

The comment after~\cite[12.22]{Dales2010bas}
asks whether there is an amenable semigroup $S$ and an ultrafilter in $\beta S$
that determines the topological centre of $\mathsf{M}(\beta S)$.
In this section we exhibit several examples of a slightly weaker property:
An amenable group $G$ and an ultrafilter in $\beta G$ that
determines the topological centre of $\beta G$;
that is, a DTC ultrafilter.
The first such example is the group $\FINSYM$ of finite permutations of $\omega$.
This group is ICC~\cite[Ex.8.3]{Ceccherini2010cag}, hence again DTC(1) by~\ref{cor:icc}.
More generally we obtain other subgroups of $\FINSYM$ that are DTC(1):

\begin{example}
    \label{ex:finsym}
\emph{Subgroups of $\FINSYM$ that act transitively on $\omega$.}

Let $G$ be a subgroup of $\FINSYM$ that acts transitively on $\omega$.
We shall prove that $G$ is ICC and therefore DTC(1).

For $x\iin\FINSYM$ denote by $\SUPP(x)$ the support of $x$.

Take any $y\iin G\setminus\{e_G\}$ and finite $F\subseteq G$ for which $y\iin F$.
There are $a,b\iin\omega$ such that $y(a)=b\neq a$.
By transitivity there is $x\iin G$ such that $x(a)\not\in\bigcup_{z\in F} \SUPP(z)$.
Then $xyx^{-1}(x(a))=x(b)\neq x(a)$, hence $x(a)\iin\SUPP(xyx^{-1})$,
hence $xyx^{-1}\not\in F$.
Thus $[y]_G$ is infinite.
\end{example}

\begin{example}
    \label{ex:metab}
\emph{A finitely generated metabelian group of exponential growth
and generalizations.}

Let $R$ be a countable infinite integral domain,
and $P$ an infinite multiplicative subgroup of $R$.
Let $G$ be the set $P\times R$ with multiplication defined by
\[
(x,r)(y,s) := (xy, r + s x )
\quad\text{for}\quad x,y \iin P, r,s \iin R .
\]
The mapping
\[
(x,r) \mapsto
\begin{pmatrix}
x& r \\
0& 1
\end{pmatrix}
\]
is an isomorphism between $G$ and a group of $2\times 2$ matrices
with the usual matrix multiplication.
A particular instance, in which $R$ is the ring of dyadic rationals
and $P$ is the multiplicative group of integer powers of 2,
is a metabelian group with two generators and
exponential growth~\cite[6.7.1]{Ceccherini2010cag}.

Write $0:=0_R$ and $1:=1_R$ and note that $e_G=(1,0)$ and
$(x,r)^{-1}=(x^{-1},-rx^{-1})$.

We shall prove that $G$ is ICC and therefore DTC(1).
Take any $(y,s)\iin G\setminus\{e_G\}$ and finite $F\subseteq G$.
Write $S:=\{t\iin R \mid (y,t) \iin F \}$.
By cancellability in $R$ we get:
\begin{itemize}
\item
If $s\neq 0$ then there exists $x\iin P$ such that
$sx\not\in S$. In that case let $r:=0$.
\item
If $s=0$ then $y\neq 1$, and there exists $r\iin R$ such that $r(1-y)\not\in S$.
In that case let $x:=1$.
\end{itemize}
Thus in both cases there exists $(x,r)\iin G$ such that
$r+sx-ry \not\in S$, hence
\[
(x,r)(y,s)(x,r)^{-1} = (y,r + sx -ry) \not\in F .
\]
That proves $[(y,s)]_G$ is infinite.
\end{example}

\begin{example}
    \label{ex:heisenberg}
\emph{Discrete Heisenberg group and generalizations.}

Let $R$ be a countable infinite integral domain.
Let $G$ be $R\times R\times R$ with the multiplication
\[
(a,b,c)(p,q,r) := (a+p,b+q,c+r+aq)
\quad\text{for}\quad a,b,c,p,q,r \iin R.
\]
The mapping
\[
(a,b,c) \mapsto
\begin{pmatrix}
1& a& c \\
0& 1& b \\
0& 0& 1
\end{pmatrix}
\]
is an isomorphism between $G$ and a group of $3\times 3$ matrices
with the usual matrix multiplication.
For the special case $R=\Zint$, the ring of integers,
this is the \emph{discrete Heisenberg group}.
In that case $G$ is finitely generated and nilpotent,
hence has no ICC quotients by the Duguid--McLain theorem~\cite{Frisch2018nvn}.
Nevertheless Theorem~\ref{th:FCinf} applies to $G$, as will now be shown,
so that $G$ is DTC(1).

Write $0:=0_R$ and $1:=1_R$ and note that $e_G=(0,0,0)$
and $(a,b,c)^{-1} = (-a,-b,ab-c)$.

Put $V:=\{(0,0,c) \mid c\iin R \}$.
We shall prove that $\FINC(G)=V$.

Clearly $[(0,0,c)]_G=\{(0,0,c)\}$ for every $c\iin R$,
hence $V\subseteq \FINC(G)$.
Take any $(p,q,r)\not\in V$ and finite $F\subseteq G$.
Write $S:=\{t\iin R \mid (p,q,t) \iin F \}$.
By cancellability in $R$ we get:
\begin{itemize}
\item
If $p\neq 0$ then there exists $b\iin R$ such that
$r-bp\not\in S$. In that case let $a:=0$.
\item
If $p=0$ then $q\neq 0$,
and there exists $a\iin R$ such that
$r+aq\not\in S$. In that case let $b:=0$.
\end{itemize}
Thus in both cases there exists $(a,b,0)\iin G$ such that $r+aq-bp\not\in S$, hence
\[
(a,b,0)(p,q,r)(a,b,0)^{-1}
= (p,q,r+aq-bp) \not\in F .
\]
That proves $[(p,q,r)]_G$ is infinite.
Thus $\FINC(G)=V$,
and so $\card{G/\FINC(G)}=\aleph_0$.
\end{example}

\begin{example}
    \label{ex:redprod}
\emph{Reduced power of a finite group.}

Let $H$ be a finite group, and $I$ an infinite index set.
Let $H^I$ be the product group, and
$G \subseteq H^I$ the \emph{reduced product},
i.e. the subgroup of those $h=(h_i)_{i\in I}\iin H^I$ for which $h_i \neq e_H$
for only finitely many coordinates $i$.
Then $G$ satisfies the assumption of Theorem~\ref{th:suff} with $n=\card{H}$,
hence it is DTC(2).
If $H$ is not abelian then $G$ is not virtually BFC.
\end{example}


\section{Open problems}

In view of Corollary~\ref{cor:virtBFC} it is natural to ask
\begin{question}
Is it true that every infinite (or at least every countable infinite) FC group
is DTC(2)?
\end{question}
A positive answer would yield an improvement of Corollary~\ref{cor:vnilp}:
It would then follow that a countable infinite group is DTC(2)
if and only if it is virtually FC.
However, as mentioned in the introduction,
we do not even know if the countable DTC(1) groups form a projective set.
If the answer to the following question is positive
then DTC(1) is analytic.
\begin{question}
Does Condition~\ref{th:countable:condd} of Theorem~\ref{th:countable}
imply Condition~\ref{th:countable:conda}?
\end{question}

The results in section~\ref{sec:countable} are specific to countable groups.
That raises
\begin{question}
Which results in section~\ref{sec:countable} generalize to uncountable groups?
\end{question}


\noindent
Jan Pachl \\
Toronto, Ontario \\
Canada   \\

\noindent
Juris Stepr\={a}ns   \\
Department of Mathematics and Statistics    \\
York University   \\
Toronto, Ontario \\
Canada

\end{document}